\documentclass[12pt]{amsart}

\usepackage{amsfonts,amssymb,enumerate,color,bm,hhline,makecell,array}
\usepackage{array}
\usepackage{mathrsfs}
\usepackage{comment}
\usepackage[all,ps,cmtip]{xy} 
\usepackage[normalem]{ulem}
\usepackage[colorlinks=true,citecolor=blue, urlcolor=blue, linkcolor=blue]{hyperref}
\usepackage{amscd}
\usepackage[shortlabels]{enumitem}
\usepackage{tensor}
\usepackage{extarrows}
\usepackage{mathtools}
\usepackage{tikz}
\usetikzlibrary{shapes}
\usetikzlibrary{arrows}
\usetikzlibrary{positioning}
\usetikzlibrary{calc}
\usepackage{tkz-euclide}
\usepackage{lipsum}
\usepackage{mwe}
\usepackage{wrapfig}
\usepackage{mathrsfs}

\def\C{\ensuremath{\mathbb{C}}}

\def\P{\ensuremath{\mathbb{P}}}
\def\Q{\ensuremath{\mathbb{Q}}}
\def\R{\ensuremath{\mathbb{R}}}
\def\Z{\ensuremath{\mathbb{Z}}}

\def\cA{\ensuremath{\mathcal A}}

\def\cC{\ensuremath{\mathcal C}}
\def\cD{\ensuremath{\mathcal D}}

\def\cF{\ensuremath{\mathcal F}}

\def\cH{\ensuremath{\mathcal H}}

\def\cM{\ensuremath{\mathcal M}}

\def\cO{\ensuremath{\mathcal O}}
\def\cP{\ensuremath{\mathcal P}}

\def\cT{\ensuremath{\mathcal T}}

\def\cW{\ensuremath{\mathcal W}}

\def\cZ{\ensuremath{\mathcal Z}}

\def\ch{\mathop{\mathrm{ch}}\nolimits}

\newcommand{\CHn}[1]{\mathop{\mathrm{CH}^{#1}_{\mathrm{num}}}\nolimits}

\def\Coh{\mathop{\mathrm{Coh}}\nolimits}

\def\Db{\mathop{\mathrm{D}^{\mathrm{b}}}\nolimits}
\def\deg{\mathop{\mathrm{deg}}\nolimits}

\def\Ext{\mathop{\mathrm{Ext}}\nolimits}

\def\Hom{\mathop{\mathrm{Hom}}\nolimits}

\def\Ker{\mathop{\mathrm{Ker}}\nolimits}

\def\Ku{\mathcal{K}u}

\def\NS{\mathop{\mathrm{NS}}\nolimits}

\def\Pic{\mathop{\mathrm{Pic}}\nolimits}
\def\Pos{\mathop{\mathrm{Pos}}}

\def\rk{\mathop{\mathrm{rk}}}

\def\Stab{\mathop{\mathrm{Stab}}\nolimits}

\def\into{\ensuremath{\hookrightarrow}}

\def\blank{\underline{\hphantom{A}}}

\theoremstyle{plain}
\newtheorem{Thm}{Theorem}[section]
\newtheorem{Prop}[Thm]{Proposition}

\newtheorem{Cor}[Thm]{Corollary}

\newtheorem{Ques}[Thm]{Question}

\newtheorem{Con}[Thm]{Conjecture}
\newtheorem*{Ques*}{Question}

\theoremstyle{definition}
\newtheorem{Def}[Thm]{Definition}
\newtheorem{Rem}[Thm]{Remark}

\newtheorem{Ex}[Thm]{Example}

\numberwithin{equation}{section}

\usepackage{todonotes}

\makeatletter
\@namedef{subjclassname@2020}{%
  \textup{2020} Mathematics Subject Classification}
\makeatother

\begin{document}

\title[Wall-crossing in algebraic geometry]{The unreasonable effectiveness of wall-crossing in algebraic geometry}

\author{Arend Bayer}
\address{\parbox{0.9\textwidth}{University of Edinburgh\\[1pt]
School of Mathematics and Maxwell Institute\\[1pt]
James Clerk Maxwell Building\\[1pt]
Peter Guthrie Tait Road, Edinburgh, EH9 3FD, United Kingdom
\vspace{1mm}}}
\email{arend.bayer@ed.ac.uk}

\author{Emanuele Macr\`i}
\address{\parbox{0.9\textwidth}{Universit\'e Paris-Saclay\\[1pt]
CNRS, Laboratoire de Math\'ematiques d'Orsay\\[1pt]
Rue Michel Magat, B\^at. 307, 91405 Orsay, France
\vspace{1mm}}}
\email{emanuele.macri@universite-paris-saclay.fr}

\thanks{This work was partially supported by EPSRC grant EP/R034826/1, and the ERC Grants ERC-2018-CoG-819864-WallCrossAG and ERC-2020-SyG-854361-HyperK}

\begin{abstract}
We survey applications of Bridgeland stability conditions in algebraic geometry and discuss open questions for future research.
\end{abstract}

\maketitle


\section{Introduction}\label{sec:intro}

Bridgeland stability conditions and wall-crossing have provided answers to many questions in algebraic geometry a priori unrelated to derived categories, including  hyper-K\"ahler varieties---their rational curves, their birational geometry, their automorphisms, and their moduli spaces---, Brill--Noether questions, Noether--Lefschetz loci, geometry of cubic fourfolds, or higher rank Donaldson--Thomas theory.
Our goal is to answer the question: why? What makes these techniques so effective, and what exactly do they add beyond, for example, classical vector bundle techniques?

The common underlying strategy can be roughly summarized as follows.
For each stability condition on a derived category $\Db(X)$ of an algebraic variety $X$ and each numerical class, moduli spaces of semistable objects in $\Db(X)$ exist as proper algebraic spaces. 
This formalism includes many previously studied moduli spaces: moduli spaces of Gieseker- or slope-stable sheaves, of stable pairs, or of certain equivalence classes of rational curves in $X$.
The set of stability conditions on $\Db(X)$  has the structure of a complex manifold; when we vary the stability condition, stability of a given object  only changes when we cross the walls of a well-defined wall and chamber structure.

The typical ingredients when approaching a problem with stability conditions are the following.
\begin{itemize}[leftmargin=!,labelwidth=\widthof{(point of interest)}, align=left]
    \item[(large volume)] There is a point in the space of stability conditions where stable objects have a ``classical'' interpretation, e.g.~as Gieseker-stable sheaves.
    \item[(point of interest)] There is a point in the space of stability conditions where stability has strong implications, e.g.~vanishing properties, or even there is no semistable object of a given numerical class.
    \item[(wall-crossing)] It is possible to analyze the finite set of walls between these two points, and how stability changes when crossing each wall.
\end{itemize}
In general, it is quite clear from the problem which are the points of interest, and the main difficulty consists of analyzing the walls.
In the ideal situation, which leads to sharp exact results, these walls can be characterized purely numerically; there are only a few such ideal situations, K3 surfaces being one of them.
Otherwise, the study of walls can get quite involved, even though there are now a number of more general results available, e.g.~a wall-crossing formula for counting invariants arising from moduli spaces.

We illustrate the case of K3 surfaces, or more generally \emph{K3 categories}, in Section~\ref{sec:Constructions1}, with applications to hyper-K\"ahler varieties, to Brill--Noether theory of curves, and to the geometry of special cubic fourfolds.
The study of other surfaces or the higher dimensional case becomes  more technical, and the existence of Bridgeland stability conditions is not yet known in full generality. There are weaker notions of stability, which in the threefold case already lead to striking results.
We give an overview of this circle of ideas in Section~\ref{sec:threefolds}, along with three applications related to curves.
We give a brief introduction to stability conditions in Section~\ref{sec:DerivedCategoryStability}, and pose some questions for future research in Section~\ref{sec:Further}.

Derived categories of coherent sheaves on varieties have been hugely influential in recent years; we refer to~\cite{BO:ICM,Bri:ICM,Kuz:ICM,Toda:ICM} for an overview of the theory. Moduli spaces of sheaves on K3 surfaces have largely been influenced by~\cite{Muk:BundlesK3}; we refer to~\cite{Muk:ICM} for an overview of applications of these techniques and to~\cite{Huy:ICM} for the higher-dimensional case of hyper-K\"ahler manifolds.
For the original motivation from physics, we refer to~\cite{Kon:ICM,Dou:ICM}.

Our survey completely omits the quickly developing theory and applications of stability conditions on Kuznetsov components of Fano threefolds. We also will not touch the rich subject of extra structures on spaces of stability conditions, developed for example in the  foundational papers~\cite{BS:QuadraticDifferentials,Bri:JoyceStructures}; we also refer to~\cite{Smi:ICM} for a symplectic perspective.


\section{Stability conditions on derived categories}\label{sec:DerivedCategoryStability}

Recall slope-stability for sheaves on an integral projective curve $C$: we set  
\[
\mu(E) = \frac{\deg E}{\rk E} \in (-\infty, +\infty], \quad \text{visualized by} \quad
Z(E) = - \deg E + i\, \rk E,
\]
and call a sheaf \emph{slope-semistable} if every subsheaf $F \subset E$ satisfies $\mu(F) \le \mu(F)$.
The set of semistable sheaves of fixed rank and degree is \emph{bounded} and can be parameterized by a projective moduli space.
Moreover, semistable sheaves generate $\Coh(C)$, in the sense that every sheaf $E$ admits a (unique and functorial) \emph{Harder-Narasimhan (HN) filtration}
\[ 0 = E_0 \subset E_1 \subset E_2 \subset \dots \subset E_m = E \]
with $E_l/E_{l-1}$ semistable, and $\mu(E_1/E_0) > \mu(E_2/E_1) > \dots > \mu(E_m/E_{m-1})$.

How to generalise this to a variety $X$ of dimension $n \geq 2$? Given a polarisation $H$, one can define the slope $\mu_H$ using $H^{n-1}\cdot \ch_1(E)$ as the degree. To distinguish e.g.~the slope of the structure sheaf $\cO_X$ from that of an ideal sheaf $\mathscr{I}_x \subset \cO_X$ for $x \in X$, we can further refine the notion of slope and use lower degree terms of the Hilbert polynomial
$p_E(m) = \chi(E(mH))$ as successive tie breakers; this yields Gieseker-stability.

One of the key insights in Bridgeland's notion of stability conditions  introduced in~\cite{Bridgeland:Stab} 
is that instead we can in fact still use a notion of slope-stability, defined as the quotient of ``degree'' by ``rank''. The price we have to pay is to replace $\Coh(X)$ by another abelian subcategory $\cA$ of the bounded derived category $\Db(X)$ of coherent sheaves on $X$, and to generalise the notions of ``degree'' and ``rank'' (combined into a central charge $Z$ as above). 

To motivate the definition, consider again 
slope-stability for a curve $C$. First, 
for $\phi \in (0, 1]$, let $\cP(\phi) \subset \Coh(C) \subset \Db(C)$ be the category of slope-semistable sheaves $E$ with $Z(E) \in \R_{>0}\cdot e^{i\pi\phi}$, i.e., of slope $\mu(E) = - \cot(\pi \phi)$, and let $\cP(\phi + n) = \cP(\phi)[n]$ for $n \in \Z$ be the set of 
\emph{semistable objects of phase $\phi + n$}.  Every complex $E \in \Db(C)$ has a filtration into its cohomology objects $\cH^l(E)[-l]$. We can combine this with the classical HN filtration of $\cH^l(E)[-l]$ for each $l$ to obtain a finer filtration for $E$ where every filtration quotient is semistable, i.e., an object of $\cP(\phi)$ for $\phi \in \R$. The properties of this structure are axiomatised by conditions \eqref{eq:Bridgeland1}--\eqref{eq:HN} in Definition \ref{def:Bridgeland} below. But crucially it can always be obtained from slope-stability in an abelian category $\cA$; we just have to generalise the setting $\cA \subset \Db(\cA)$ to $\cA \subset \cD$ being the ``heart of a bounded t-structure'' in a triangulated category.

When combined with the remaining conditions in Definition~\ref{def:Bridgeland}, the main payoffs are the strong deformation and wall-crossing properties of Bridgeland stability conditions. Given any small deformation of ``rank'' and ``degree'' (equivalently, of the central charge $Z$), we can accordingly adjust the abelian category $\cA$ (or, equivalently, the set of semistable objects $\cP$) and obtain a new stability condition.  Along such a deformation, moduli spaces of semistable objects undergo very well-behaved wall-crossing transformations.

\subsection{Bridgeland stability conditions}\label{subsec:Bridgeland}

We now consider more generally an  \emph{admissible subcategory} $\cD$ of $\Db(X)$ for a smooth and proper variety $X$ over a field $k$: a full triangulated subcategory whose inclusion $\cD \hookrightarrow \Db(X)$ admits both a left and a right adjoint. For instance, $\cD=\Db(X)$; otherwise
we think of $\cD$ as a smooth and proper \emph{non-commutative variety}.

We fix a finite rank free abelian group $\Lambda$ and a group homomorphism
\[
v\colon K_0(\cD)\to \Lambda
\]
from the Grothendieck group $K_0(\cD)$ of $\cD$ to $\Lambda$.

\begin{Def}\label{def:Bridgeland}
A \emph{Bridgeland stability condition} on $\cD$ with respect to $(v,\Lambda)$ is a pair $\sigma=(Z,\cP)$ where
\begin{itemize}
    \item $Z\colon \Lambda \to \C$ is a group homomorphism, called \emph{central charge}, and
    \item $\cP=\bigl(\cP(\phi)\bigr)_{\phi\in\R}$ is a collection of full additive subcategories $\cP(\phi)\subset \cD$
\end{itemize}
satisfying the following conditions:
\begin{enumerate}
    \item\label{eq:Bridgeland1} for all nonzero $E\in\cP(\phi)$ we have $Z(v(E))\in \R_{>0}\cdot e^{i\pi\phi}$;
    \item\label{eq:Bridgeland2} for all $\phi\in\R$ we have $\cP(\phi+1)=\cP(\phi)[1]$;
    \item\label{eq:Bridgeland3} if $\phi_1>\phi_2$ and $E_j\in\cP(\phi_j)$, then $\Hom(E_1,E_2)=0$;
    \item\label{eq:HN}  (Harder--Narasimhan filtrations) for all nonzero $E\in\cD$ there exist real numbers $\phi_1>\phi_2>\dots>\phi_m$ and a finite sequence of morphisms
    \[
    0=E_0 \xrightarrow{s_1} E_1 \xrightarrow{s_2}\dots \xrightarrow{s_m} E_m=E
    \]
    such that the cone of $s_l$ is a non-zero object of $\cP(\phi_l)$;
    \item\label{eq:SupportProperty} (support property) there exists a quadratic form $Q$ on $\Lambda_{\R} = \Lambda \otimes \R$ such that
    \begin{itemize}
        \item the kernel of $Z$ is negative definite with respect to $Q$, and
        \item for all $E\in\cP(\phi)$ for any $\phi$ we have $Q(v(E))\geq0$;
    \end{itemize}
    \item\label{eq:openness} (openness of stability) the property of being in $\cP(\phi)$ is open in families of objects in $\cD$ over any scheme;
    \item\label{eq:boundedness} (boundedness) objects in $\cP(\phi)$ with fixed class $v\in\Lambda$ are parameterized by a $k$-scheme of finite-type.
\end{enumerate}
\end{Def}

An object of the subcategory $\cP(\phi)$ is called \emph{$\sigma$-semistable} of phase $\phi$, and  \emph{$\sigma$-stable} if it admits no non-trivial suboject in $\cP(\phi)$.
The set of Bridgeland stability conditions on $\cD$ is denoted by $\Stab(\cD)$, where we omit the dependence on $(v,\Lambda)$ from the notation.

Conditions~\eqref{eq:Bridgeland1}--\eqref{eq:HN} form the original definition in~\cite{Bridgeland:Stab} and ensure we have a notion of slope-stability.
The support property is necessary to show that stability conditions vary continuously (see Theorem~\ref{thm:BridgelandDeformationThm} below) and admit a well-behaved wall and chamber structure: fundamentally, this is due to the simple linear algebra consequence that given $C > 0$, there are only finitely many classes $w \in \Lambda$ of semistable objects with $|Z(w)| < C$.
Conditions~\eqref{eq:openness} and~\eqref{eq:boundedness} were introduced in \cite{Toda:K3,Toda:ICM}, with similar versions appearing previously in \cite[Section 3]{Kontsevich-Soibelman:stability}; they guarantee the existence of moduli spaces of semistable objects.

\begin{Thm}[Bridgeland Deformation Theorem]\label{thm:BridgelandDeformationThm}
The set $\Stab(\cD)$ has the structure of a complex manifold, such that the natural map
\[
\cZ\colon \Stab(\cD) \to \Hom(\Lambda, \C), \quad (Z, \cP) \mapsto Z
\]
is a local isomorphism at every point. 
\end{Thm}
For conditions \eqref{eq:Bridgeland1}--\eqref{eq:SupportProperty}, this is a reformulation of Bridgeland's main result \cite[Theorem 1.2]{Bridgeland:Stab}. It says that $\sigma = (Z, \cP)$ can be deformed uniquely given a small deformation of $Z \leadsto Z'$, roughly as long as $Z'(E) \neq 0$ remains true for all $\sigma$-semistable objects $E$. (More precisely, any path where $Q$ remains negative definite on $\Ker Z'$ can be lifted uniquely to a path in $\Stab(\cD)$.)
With the additional conditions~\eqref{eq:openness} and~\eqref{eq:boundedness}, Theorem~\ref{thm:BridgelandDeformationThm} was proved in \cite[Theorem 3.20]{Toda:K3} and \cite[Section 4.4]{PT15:bridgeland_moduli_properties}, where the most difficult aspect  is to show that openness of stability is preserved under deformations.

The theory has been developed over an arbitrary base scheme in~\cite{families}.
A stability condition over a base is the datum of a stability condition on each fiber, such that families of objects over the base have locally constant central charges, satisfy openness of stability, and a global notion of HN filtration after base change to a one-dimensional scheme; moreover, we impose a global version of the support property and of boundedness.
An analogue of Theorem~\ref{thm:BridgelandDeformationThm} holds; differently to the absolute case, assuming the support property is not enough and the proof requires the additional conditions~\eqref{eq:openness} and~\eqref{eq:boundedness}.

The construction of Bridgeland stability conditions is discussed in Section~\ref{sec:threefolds}; in particular, they exist on surfaces and certain threefolds.

\subsection{Stability conditions as polarizations}\label{subsec:ModuliSpaces}

It was first suggested in the arXiv version of \cite{Bridgeland:K3} to think of $\sigma$ as a polarization of the non-commutative variety $\cD$. We now review some results partly justifying this analogy. A polarisation of a variety $X$ by an ample line bundle $H$ gives projective moduli spaces of $H$-Gieseker-stable sheaves; the following two results provide an analogue.

\begin{Thm}[Toda, Alper, Halpern-Leistner, Heinloth] \label{thm:modulispaces}
Given $\sigma \in \Stab(\cD)$ and $v \in \Lambda$, there is a finite type Artin stack $\cM_{\sigma}(v)$ of $\sigma$-semistable objects $E$ with $v(E) = v$ and fixed phase. In characteristic zero, it has a proper \emph{good moduli space} $M_\sigma(v)$ in the sense of Alper.
\end{Thm}

\begin{proof}
The existence as Artin stack is~\cite[Theorem 3.20]{Toda:K3}, while the existence of a good moduli space is proven in \cite[Theorem 7.25]{AHLH:good_moduli}; see also \cite[Theorem 21.24]{families}.
\end{proof}

\begin{Thm} \label{thm:projectivemoduli}
The algebraic space $M_\sigma(v)$ admits a Cartier divisor $\ell_\sigma$ that has strictly positive degree on every curve. In characteristic zero, if  $M_\sigma(v)$ is smooth, or more generally if it has $\Q$-factorial log-terminal singularities, then $M_\sigma(v)$ is projective.
\end{Thm}

\begin{proof}
The existence of the Cartier divisor and its properties is the \emph{Positivity Lemma} in~\cite{BM:proj}, see also~\cite[Theorem 21.25]{families}. The projectivity follows from \cite[Corollary 3.4]{Villalobos-Paz:moishezonProjectivity}.
\end{proof}

As studied extensively in Donaldson theory in the 1990s, the Gieseker-moduli spaces may change as $H$ crosses walls in the ample cone.

\begin{Thm}\label{thm:wallchambers}
Fix a vector $v \in \Lambda$. Then there exists a locally finite union $\cW_v$ of real-codimension one submanifolds in $\Stab(\cD)$, called walls, such that on every connected component $\cC$  of the complement $\Stab(\cD) \setminus \cW_v$, called a chamber,  the moduli space $M_\sigma(v)$ is independent of the choice $\sigma \in \cC$.
\end{Thm}

Theorem~\ref{thm:wallchambers} follows from the results in~\cite[Section 9]{Bridgeland:K3}; see also \cite[Proposition 2.8]{Toda:K3} and \cite[Proposition 3.3]{BM:localP2}.
The set $\cW_v$ consists of stability conditions for which there exists an exact triangle $A \to E \to B$ of semistable objects of the same phase with $v(E) = v$, but $v(A)$ not proportional to $v$. Locally the wall is defined by $Z(A)$ being proportional to $Z(E)$, and the objects $E$ is unstable on the side where $\mathrm{arg}(Z(A)) > \mathrm{arg}(Z(E))$; often it is stable near the wall on the other side, e.g.~when $A$ and $B$ are \emph{stable} and the extension is non-trivial. The support property~\eqref{eq:SupportProperty} is again crucial in the proof of  Theorem~\ref{thm:wallchambers}: it constrains the classes $a= v(A), b = v(B)$ involved in a wall, and locally that produces a finite set.

Sometimes, one can describe $\cW_v$ completely, namely when we know which of the moduli spaces $M_\sigma(a)$ and $M_\sigma(b)$ are non-empty.

\subsection{K3 categories}\label{subsec:K3}

Such descriptions of $\cW_v$ have been particularly powerful
in the case of K3 categories; it has also been carried out completely for $\Db(\P^2)$, where the answer is  more involved~\cite{CHW:effectiveP2, ChunyiXiaolei:MMPP2}.
For this section, we work over the complex numbers and we let $\cD$ be one of
\begin{enumerate}
    \item\label{enum:K3category1} the derived category $\cD = \Db(S)$ of a smooth projective K3 surface, or
    \item\label{enum:K3cubic} the \emph{Kuznetsov component} 
    \[
    \cD = \Ku(Y) = \cO_Y^\perp \cap \cO_Y(1)^\perp \cap \cO_Y(2)^\perp \subset \Db(Y)
    \]
    of the derived category of a smooth cubic fourfold $Y$, or
    \item\label{enum:K3GM} the Kuznetsov component of a Gushel--Mukai fourfold defined in~\cite{KP:GMcat}.
\end{enumerate}
In~\eqref{enum:K3category1} we can also allow a Brauer twist; one expects further examples of Kuznetsov components of Fano varieties where similar results hold.
In all these cases, $\cD$ is a Calabi-Yau-2 category: there is a functorial isomorphism $\Hom(E, F) = \Hom(F, E[2])^\vee$ for all $E, F \in \cD$.
Moreover, it has an associated integral weight two Hodge structure $H^*(\cD, \Z)$ with an even pairing $(\blank, \blank)$; in the case of a K3 surface, 
$H^*(\Db(S), \Z) = H^*(S, \Z)$ with $H^0$ and $H^4$ considered to be $(1, 1)$-classes; in the other cases, the underlying lattice is the same, and after the initial indirect construction in  \cite{AT:CubicFourfolds} there is now an intrinsic construction based on the topological K-theory and Hochschild homology of $\cD$ \cite{IHC-CY2}. There is a Mukai vector $v \colon K_0(\cD) \to H^{1, 1}(\cD, \Z)$ satisfying $\chi(E, F) = - (v(E), v(F))$ for all $E, F \in \cD$.

In all three cases, there is a main component $\Stab^\dagger(\cD)\subset\Stab(\cD)$ with an effective version of Theorem~\ref{thm:BridgelandDeformationThm} for $\Lambda = H^{1, 1}(\cD, \Z)$: the map $\cZ$ is a covering of an explicitly described open subset of $\Hom(\Lambda, \C)$, see \cite{Bridgeland:K3} for case~\eqref{enum:K3category1}, \cite{BLMS:inducing, families} for \eqref{enum:K3cubic}, and \cite{PPZ:GM4folds} for  \eqref{enum:K3GM}.

Now consider a family of such K3 categories given by a family of K3 surfaces or Fano fourfolds over a base scheme. In this case, Mukai's classical deformation argument  applies: every stable object $E$ in a given fibre is simple, i.e., it satisfies $\Hom(E, E) = \C$, and so $\Ext^2(E, E) = \C$ by Serre duality; therefore the obvious obstruction to extending $E$ across the family, namely that $v(E)$ remains a Hodge class, is the only one. Extending such deformation arguments to $\cD$ was the original motivation for introducing stability conditions for families of non-commutative varieties, see \cite[Section 31]{families}. This setup allows us to deduce non-emptiness of moduli spaces from the previously known case of K3 surfaces (it also simplifies the previous classical argument for Gieseker-stability on K3 surfaces by reduction to elliptically fibered K3s, see \cite{Bottini:Yoshioka}), which leads to the following result. 

\begin{Thm}[Mukai, Huybrechts, O'Grady, Yoshioka, Toda, \cite{BM:proj, families, PPZ:GM4folds}]\label{thm:Yoshioka}
Let $v \in H^{1, 1}(\cD, \Z)$ be  primitive, and $\sigma \in \Stab^\dagger(\cD)$ be generic. Then $M_\sigma(v)$ is non-empty if and only if $v^2:=(v,v) \ge -2$; in this case, it is a smooth projective irreducible holomorphic symplectic (IHS) variety.
\end{Thm}

More precisely, $M_\sigma(v)$ is of $\mathrm{K3}^{[n]}$-type, where $n=(v^2+2)/2$, i.e., it is deformation equivalent to the Hilbert scheme of $n$ points on a K3 surface (see~\cite{GrossHuybrechtsJoyce,Debarre:survey} for the basic theory of irreducible holomorphic symplectic varieties).
If $v^2\geq2$, the Mukai morphism
\[
\vartheta\colon H^2(M_\sigma(v),\Z)\to H^*(\cD,\Z)
\]
induced by a (quasi-)universal family gives an identification of $H^2(M_\sigma(v),\Z)$ with $v^\perp$.
If $v^2=0$, then $M_\sigma(v)$ is a K3 surface and $H^2(M_\sigma(v),\Z)$ is identified with $v^\perp/v$.


Knowing exactly which semistable objects exist then allows us to describe exactly when we are on a wall. While a complete result as in \cite[Theorem 5.7]{BM:walls} also needs to treat essential aspects of the wall-crossing behaviour, the basic result is simple to state: 

\begin{Thm}[\cite{BM:walls}]\label{thm:WallsK3}
Let $v \in H^{1, 1}(\cD, \Z)$ be a primitive class. Then $\sigma = (Z, \cP) \in \Stab^\dag(\cD)$ lies on a wall for $v$ if and only if there exist classes $a, b \in H^{1, 1}(\cD, \Z)$ with $v = a+b$, $a^2, b^2 \ge -2$ and $Z(a), Z(b)$ are positive real multiples of $Z(v)$.
\end{Thm}

And the fundamental reason is similarly simple to explain: by Theorem \ref{thm:Yoshioka}, this allows for the existence of extensions
\begin{equation}\label{eq:WallsK3}
0 \to A \to E \to B \to 0
\end{equation}
where $v(A) = a, v(E) = v, v(B) = b$ and $A, E, B$ are all semistable of the same phase. For stronger results, we  need to know when such $E$ can become \emph{stable} near the wall.


\section{Constructions based on K3 categories}\label{sec:Constructions1}

In this section we present three applications of stability conditions on K3 categories, to irreducible holomorphic symplectic varieties, to  curves, and to cubic fourfolds.

\subsection{Curves in irreducible holomorphic symplectic manifolds}\label{subsec:walls}

Let $M$ be a smooth projective irreducible holomorphic symplectic (IHS) variety of $\mathrm{K3}^{[n]}$-type, with $n\geq2$.
We let $q_M$ be the Beauville--Bogomolov--Fujiki quadratic form on $H^2(M,\Z)$.
By~\cite[Section~3.7.1]{Debarre:survey}, there exists a canonical extension
\[
\vartheta_M \colon (H^2(M,\Z),q_M) \hookrightarrow \widetilde{\Lambda}_M
\]
of lattices and weight-2 Hodge structures, where the lattice $\widetilde{\Lambda}_M$ is isometric to the extended K3 lattice $U^{\oplus 4}\oplus E_8(-1)^{\oplus 2}$.
Let us denote by $v\in\widetilde{\Lambda}_M$ a generator of $\vartheta(H^2(M,\Z))^\perp$: it is of type $(1,1)$ and square $v^2=2n-2$.
The lattice $\widetilde{\Lambda}_M$ is called the \emph{Markman--Mukai} lattice associated to $M$.
If $M=M_\sigma(v)$, for a stability condition $\sigma\in\Stab^\dagger(\Db(S))$ on a K3 surface $S$, then $\widetilde{\Lambda}_M=H^*(S,\Z)$ with the Mukai pairing, the notation for the vector $v$ is coherent, and $\vartheta_M$ is the Mukai morphism mentioned after Theorem~\ref{thm:Yoshioka}.

We let $\Pos(M)$ be the connected component of the positive cone of $M$ containing an ample divisor class:
\[
\Pos(M):=\left\{ D\in H^2(M,\R)\,:\, q_M(D)>0 \right\}^+.
\]

The following result rephrases and proves a conjecture by Hassett--Tschinkel and gives a complete description of the ample cone of $M$.

\begin{Thm}\label{thm:walls}
Let $M$ be a smooth projective IHS variety of $\mathrm{K3}^{[n]}$-type.
The ample cone of $M$ is a connected component of
\[
\mathrm{Pos}(M) \setminus \bigcup_{\begin{subarray}{c}{a \in \widetilde{\Lambda}_{M}^{1,1}\ \mathrm{s.t.}}\\ {a^2\geq -2\ \mathrm{and}} \\ {0\leq (a,v) \leq v^2/2}\end{subarray}} a^\perp.
\]
\end{Thm}

Theorem~\ref{thm:walls} is proved in~\cite{BM:walls} for moduli spaces of stable sheaves/complexes on a K3 surface, and extended in~\cite{BHT:nef} to all IHS of $\mathrm{K3}^{[n]}$-type, by using deformation theory of rational curves on IHS varieties.

The approach to Theorem~\ref{thm:walls} via wall-crossing is as follows. Let $S$ be a K3 surface and $M = M_{\sigma_0}(v)$ be a moduli space  of $\sigma_0$-stable objects in $\Db(S)$, where $v\in H^{1,1}(\Db(S),\Z)$ is a primitive vector of square $v^2\geq2$. As $\sigma$ varies in the chamber $\cC$ containing $\sigma_0$, Theorem~\ref{thm:projectivemoduli} gives a family of ample divisor classes $\ell_\sigma$ in $\Pos(M)$. When $\sigma$ reaches a wall of $\cC$, as given by Theorem~\ref{thm:WallsK3}, the class $\ell_\sigma$ remains nef. On the other hand, consider an object $E$ that becomes strictly semistable on the wall, admitting an exact sequence as in \eqref{eq:WallsK3}. Varying the extension class in a line in $\P(\Ext^1(B, A))$ produces a $\P^1$ of such objects, and Theorem~\ref{thm:projectivemoduli} shows that $\ell_\sigma$ has degree zero on this curve. We have found an extremal curve and, dually, a boundary wall of the ample cone. 

\begin{figure}
\begin{center}
\scalebox{0.8}{
\begin{tikzpicture}[
  node distance = 1cm,
  bubble/.style = {draw, ellipse, minimum width = 1.5cm, minimum height = 1.0cm},
  line/.style = {-latex'},
  ]
  \node[bubble, fill=gray!5!white] (B1) {Ideal sheaves of points};
  \node[bubble, fill=gray!5!white, below = of B1] (B2) {Stable sheaves on elliptic K3s};
  \node[bubble, fill=gray!5!white, below = of B2] (B3) {Stable sheaves on K3s};
  \node[bubble, fill=gray!5!white, below = of B3] (B4) {Stable objects on K3s};
  \node[bubble, fill=gray!5!white, below = of B4] (B5) {Rat'l curves on moduli spaces};
  \node[bubble, fill=gray!5!white, below = of B5] (B6) {Rat'l curves on IHSs of $\mathrm{K3}^{[n]}$-type};
  \draw [->] (B1) -- (B2) ;
  \draw [->] (B2) -- (B3) ;
  \draw [->] (B3) -- (B4);
  \draw [->] (B4) -- (B5);
  \draw [->] (B5) -- (B6);
  \draw [->] (B1.west) to [out=-150] (B4.west);
    \draw [->] (B1.west) to [out=-150] (B3.west);
  
  \node at ($(B2)+(1.7,1)$) {\text{autoequivalences}};
   \node at ($(B2)+(-0.5,1)$) {\text{\cite{Yoshioka:Abelian}}};
  \node at ($(B3)+(1.2,1)$) {\text{deformation}};
    \node at ($(B3)+(-0.5,1)$) {\text{\cite{Yoshioka:Abelian}}};
    \node at ($(B3)+(-5.8,1)$) [align=left]{wall-crossing, \\ deformation \\ of objects \cite{Bottini:Yoshioka}};
  \node at ($(B4)+(1.25,1)$) {\text{wall-crossing}};
      \node at ($(B4)+(-0.8,1)$) {\text{\cite{Toda:K3, BM:proj}}};
  \node at ($(B5)+(1.25,1)$) {\text{wall-crossing}};
      \node at ($(B5)+(-0.4,1)$) {\text{\cite{BM:walls}}};
  \node at ($(B6)+(2.5,1)$) {deformation of rat'l curves};
      \node at ($(B6)+(-0.5,1)$) {\text{\cite{BHT:nef}}};
\end{tikzpicture}
}
\end{center}
\caption{The approach to Theorem~\ref{thm:walls}.} \label{fig:Schematic}
\end{figure}
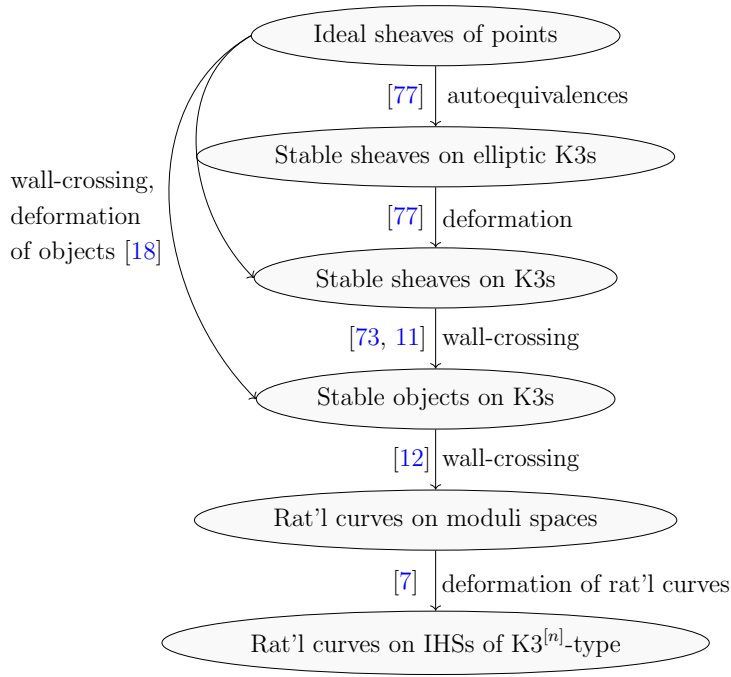

We summarize the history underlying Theorem~\ref{thm:walls} with the  diagram in Figure~\ref{fig:Schematic}. The analogue of Theorem~\ref{thm:Yoshioka} for Gieseker-stable sheaves involves a two-step argument,  using autoequivalences and deformations.
Wall-crossing techniques then imply the existence of Bridgeland stable objects on K3 surfaces, and thus Theorem~\ref{thm:Yoshioka}. As discussed above, a finer wall-crossing analysis based on Theorem~\ref{thm:WallsK3}  then produces the extremal rational curves on moduli spaces that appear implicitly as extremal curves in Theorem~\ref{thm:walls}. 
Finally, another deformation argument, involving rational curves, deduces Theorem~\ref{thm:walls} for all IHS manifolds of $\mathrm{K3}^{[n]}$-type.
Wall-crossing combined with stability conditions in families can also simplify the approach to Theorem~\ref{thm:Yoshioka}, see \cite{Bottini:Yoshioka}.

\subsection{Curves}\label{subsec:Curves}

Consider a \emph{Brill--Noether (BN) wall} in $\Stab(\Db(X))$ for a variety $X$: the structure sheaf $\cO_X$ is stable and of the same phase $\phi$ as objects $E$ of a fixed class $v$. 
Then $\cO_X$ is an object of the abelian category $\cP(\phi)$ with no subobjects; hence the evaluation map $\cO_X \otimes H^0(E) \to E$  must be injective, giving a short exact sequence
\begin{equation} \label{eq:BNses}
0 \to \cO_X \otimes H^0(E) \to E \to Q \to 0 \end{equation}
where $Q \in \cP(\phi)$ is also semistable. Applying known inequalities for Chern classes of semistable objects to 
$\ch(Q) = v - h^0(E) \ch(\cO_X)$ can directly lead to bounds on $h^0(E)$. 

This simple idea turns out to be powerful. For a K3 surface $S$, we can be more precise: applying Theorem \ref{thm:Yoshioka} to the class of $Q$, we can \emph{construct} all $E$ with given $r = h^0$ as a Grassmannian bundle $\mathrm{Gr}(r, \Ext^1(Q, \cO_S))$ of extensions over the moduli space of such $Q$. 

\begin{Cor}\label{cor:BN}
Let $S$ be a K3 surface, $v \in H^{1, 1}(S, \Z)$ primitive and $\sigma$ be a stability condition near the Brill--Noether wall for $v$. If the lattice generated by $v$ and $v(\cO_S)$ is saturated, then the locus of objects $E \in M_\sigma(v)$ with $h^0(E) = r$ has expected dimension.
\end{Cor}

In \cite{Arend:BN}, this is applied, in the case where $\Pic(S) = \Z\cdot H$, to rank zero classes of the form $v = (0, H, s)$. In this case, there are no walls between the BN wall and the large volume limit; hence Corollary \ref{cor:BN} applies in the large volume limit, and thus to zero-dimensional torsion sheaves supported on curves in the primitive linear system. This gives a variant of Lazarsfeld's proof \cite{Lazarsfeld:BN-Petri} of the Brill--Noether theorem: \emph{every} curve in the primitive linear system is Brill--Noether general.

This approach has been significantly strengthened in \cite{Soheyla:Mukai}: instead of requiring $E$ to be semistable near the Brill--Noether wall, it is sufficient to control the classes occurring in its HN filtration. A bound on $h^0$ is obtained by applying Corollary~\ref{cor:BN} to all HN filtration factors. Thus we need to consider a point near the Brill--Noether walls for all HN factors, and which is the limit point where $Z(\cO_X) \leadsto 0$. 

\begin{Prop}[{\cite[Proposition 3.4]{Soheyla:Mukai}}]\label{prop:SoheylaMukai} 
Let $S$ be a K3 surface of Picard rank one. There exists a limit point $\overline{\sigma}$ of the space of stability conditions, with central charge $\overline{Z}$, and a constant $C$, such that for (most) objects in the heart, we have
\[
h^0(E) + h^1(E) \leq C \cdot \sum_{l} \left\lvert\overline{Z}(E_l/E_{l-1})\right\rvert,
\]
where $E_0 \subset E_1 \subset \dots \subset E_m$ is the HN filtration of $E$ near $\overline{\sigma}$. 
\end{Prop}

The following application completes a program originally proposed by Mukai \cite{Mukai:BN}:

\begin{Thm}[\cite{Soheyla:Mukai,Soheyla:Mukai2}]\label{thm:SoheylaMukai}
Let $S$ be a polarised K3 surface with $\Pic S = \Z\cdot H$ and genus $g \ge 11$, and let $C \in \lvert H\rvert$. Then $S$ is the unique K3 surface containing $C$, and can be reconstructed as a Fourier--Mukai partner of a Brill--Noether locus of stable vector bundles on $C$ with prescribed number of sections.
\end{Thm}

The structure of the argument is as follows. The numerics are chosen such that there is a two--dimensional moduli space $\hat S$, necessarily a K3 surface, of stable bundles $E$ on $S$ whose restriction $E|_C$ is automatically in the Brill--Noether locus. Conversely, given a stable bundle $V$ on $C$, its push-forward $i_*V$ along $i \colon C \into S$ is stable at the large volume limit. Standard wall-crossing arguments bound its HN filtration near the limit point $\overline{\sigma}$ in 
Proposition~\ref{prop:SoheylaMukai}, which then gives a bound on $h^0(V)$. The argument also shows that equality---the Brill--Noether condition---only holds for the HN filtration $E \to E|_C = i_* V \to E(-H)[1]$, i.e., when $V$ is the restriction of a vector bundle in $\hat S$. Thus $\hat S$ is a Brill--Noether locus on $C$, and $S$ can be reconstructed as a Fourier-Mukai partner of $\hat S$.

\subsection{Surfaces in cubic fourfolds}\label{subsec:Surfaces}

Let $Y\subset\mathbb{P}^5$ be a complex smooth cubic fourfold and let $h$ be the class of a hyperplane section.
Following~\cite{Hassett:specialcubics}, we say that $Y$ is \emph{special} of discriminant $d$, and write $Y\in\mathcal{C}_d$, if there exists a surface $\Sigma\subset Y$ such that $h^2$ and $\Sigma$ span a saturated rank two lattice in $H^4(Y, \Z)$ with
\[
\mathrm{det} \begin{pmatrix}
h^4  & h^2\cdot\Sigma\\
h^2\cdot\Sigma & \Sigma^2
\end{pmatrix} = d.
\]
The locus $\mathcal{C}_d$ is non-empty if and only if $d\equiv 0,2\pmod 6$ and $d>6$; in this case, $\mathcal{C}_d$ is an irreducible divisor in the moduli space of cubic fourfolds. 

Given $d$ it is not known in general which degree $h^2\cdot \Sigma$ and self-intersection $\Sigma^2$ can be realised.
The following gives an answer for an infinite series of $d$.

\begin{Thm}[\cite{BBMP}]\label{thm:BBMP}
Let $a\geq1$ be an integer and let $d:=6a^2+6a+2$.
Let $Y$ be a general cubic fourfold in $\mathcal{C}_d$.
Then there exist surfaces $\Sigma\subset Y$ such that
\begin{itemize}
\item $\deg(\Sigma):=h^2\cdot\Sigma=1+\frac{3}{2} a(a+1)$ and $\Sigma^2 = \frac{d+\deg(\Sigma)^2}{3}$;
\item $H^*(Y,\mathscr{I}_\Sigma(a-j))=0$, for all $j=0,1,2$.
\end{itemize}
\end{Thm}

In fact, we obtain a family of such surfaces $\Sigma$. It is parameterized by (an open subset of) the K3 surface $S$ of degree $d = 6a^2 + 6a + 2$ associated to every $Y \in \cC_d$ by Hassett \cite{Hassett:specialcubics} via the Hodge structures of $Y$ and $S$.

\def\Bl{\mathrm{Bl}}

The proof of Theorem~\ref{thm:BBMP} is based on an enhancement of Hassett's Hodge-theoretic relation to the derived category: by \cite{AT:CubicFourfolds, families}, we have $\Ku(Y) = \Db(S)$. By \cite{LiPertusiZhao:TwistedCubics}, the Lehn-Lehn-Sorger-van Straten IHS eightfold $X$ associated to $Y$ of \cite{LLSvS} can be realised as a moduli space of stable objects in $\Ku(Y)$. It admits a Lagrangian embedding $Y \into X$. Markman's Torelli Theorem \cite{Eyal:survey} implies that $X$ and the Hilbert scheme $S^{[4]}$ are birational; for $a \ge 2$, and $Y \in \cC_d$ very general, we use Theorem~\ref{thm:walls} to show that the nef cone and the movable cone of $S^{[4]}$ are the same, and hence $X$ and $S^{[4]}$ are isomorphic. Now for $s \in S$ general let $\Gamma_s \subset S^{[4]}$ be the locus of subschemes containing $s$; then $\Sigma = Y \cap \Gamma_s$ is the desired surface.

We are interested in this construction because it produces many rational morphisms from $Y$ as follows. The rational map $S^{[4]} \to S^{[5]}, Z \mapsto Z \cup s$ is resolved by the blow-up $\Bl_{\Gamma_s} S^{[4]} \to S^{[5]}$.
Restricting to the cubic, we obtain an embedding $\Bl_\Sigma Y \to S^{[5]}$. We then use wall-crossing for $S^{[5]}$, interpreted as a moduli space in $\Ku(Y) = \Db(S)$, and restrict the resulting birational transformations to $\Bl_\Sigma Y$. In the case $a=2$, this recovers completely the picture described in~\cite{RussoStagliano:TrisecantFlop} and thus the rationality of all cubics in $\mathcal{C}_{38}$. Analogous constructions likely exist for arbitrary $d \equiv 2 \pmod 6$ if we replace $\Gamma_s$ with a locus of sheaves not locally free at $s$ inside a moduli space of stable sheaves.


\section{Threefolds}\label{sec:threefolds}

Stability conditions on a threefold $X$ have been constructed in a three step process. Initially we consider slope-stability. In the second step, we reinterpret slope-stability by changing the abelian subcategory of $\Db(X)$ and rotating the central charge; 
the classical Bogomolov-Gieseker inequality then allows us to  deform this to obtain \emph{tilt-stability}, which behaves much like Bridgeland stability conditions on surfaces. Finally, a conjectural generalized Bogomolov-Gieseker type inequality, Conjecture~\ref{conj:generalizedBG}, for tilt-stable objects allows one to repeat this procedure and produce actual stability conditions.

Already the second step, tilt-stability, has geometric consequences when combined with Conjecture~\ref{conj:generalizedBG}. We present three applications: to a bound for the genus of a curve on a threefold~\cite{MS:genus}, to higher rank Donaldson--Thomas theory on Calabi--Yau threefolds~\cite{FT1,FT2}, and to Clifford-type bounds for vector bundles on curves and quintic threefolds~\cite{Li:Quintic3fold}.

\subsection{The generalized Bogomolov--Gieseker inequality}\label{subsec:BG}
We will now describe the first and second step of this construction. For more details, we refer to \cite[Part V]{families}, \cite{PT15:bridgeland_moduli_properties} and \cite{BMS:abelian3folds}.
Throughout this section, we let $X$ be a smooth  projective variety of dimension~$n$, over the complex numbers unless noted otherwise, and let $H$ in $\NS(X)$ be the class of an ample divisor on $X$.

\subsubsection*{The twisted Chern character}
Let
\begin{equation}\label{eq:gamma}
\gamma := e^{-B}\cdot \left(1, 0, -\Gamma,\gamma_3,\dots,\gamma_n\right)\in\bigoplus_{l= 0}^n \CHn{l}(X)_\Q,
\end{equation}
with $B\in\NS(X)_\Q$, $\Gamma\in\CHn{2}(X)_\Q$ such that $H^{n-2}\cdot\Gamma=0$, and $\gamma_3,\dots,\gamma_n$ arbitrary.
We let
\[
\ch^\gamma := \gamma\cdot\ch \colon K_0(X) \to \bigoplus_{l= 0}^n \CHn{l}(X)_\Q
\]
be the \emph{Chern character twisted by $\gamma$}.
If $\gamma=e^{-B}$ (e.g.~in the case of surfaces) then $\ch^\gamma$ is usually denoted by $\ch^B$.
When $X$ is a threefold of Picard rank~1, $\Gamma=0$.

Let $\Lambda_H^\gamma\subset\Q^{n+1}$ be the image of the morphism $v^\gamma_H\colon K_0(\Db(X))\to \Q^{n+1}$ given by
\[
v^\gamma_H(E):=\big(H^n\cdot\ch_0^\gamma(E),H^{n-1}\cdot\ch_1^\gamma(E),\dots,H\cdot\ch_{n-1}^\gamma(E),\ch_n^\gamma(E)\big).
\]
It is a free abelian group of rank~$n+1$.
Given $v\in\Lambda_H^\gamma$, we denote by $v_l$ its $l$-th component.
For $l=0,\dots,n$, we denote by $\Lambda_{H,\leq l}^\gamma$ the sublattice of rank~$l+1$ generated by the first $l+1$ components and by $v_{\leq l}$ the corresponding truncated vector.

\subsubsection*{Slope-stability (Step 1)}
This is analogous to the curve case in Section~\ref{sec:DerivedCategoryStability}. We define the slope of a coherent sheaf $E\in\Coh(X)$ as 
\[ \mu_{H}^{\gamma}(E):=\mu(v^\gamma_H(E)) =v_1(E)/v_0(E), \quad \text{with $\mu_H^\gamma(E) = +\infty$ if $v_0(E) = 0$}.
\]
A sheaf $E$ is \emph{$\mu_{H}^{\gamma}$-semistable} if for every non-zero subsheaf $F\hookrightarrow E$, we have $\mu_{H}^{\gamma}(F)\leq\mu_{H}^{\gamma}(E/F)$.
In particular, torsion sheaves are semistable of slope $+\infty$.
Harder--Narasimhan (HN) filtrations for slope stability exist; if the sheaf has torsion, the first HN factor is the torsion part.
As before, we visualize the slope with  the (weak) stability function $Z\colon\Lambda^\gamma_H\to\C$ given by
\[
Z (v) := - v_1 + i\, v_0.
\]

\subsubsection*{The Bogomolov--Gieseker (BG) inequality}
Let us assume $n\geq2$ and define the  quadratic form $\overline{\Delta}$ on $\Lambda_{H,\leq 2}^\gamma$ of signature $(2,1)$ by
\[
\overline{\Delta} (v) := v_1^2 - 2\,v_0\,v_2.
\]
The following result is a consequence of~\cite{Reid:Bog,Gieseker:Bog,Bogomolov:Ineq} and the Hodge index theorem:

\begin{Thm}[BG inequality]\label{thm:BG}
Let $E\in\Coh(X)$ be a $\mu_{H}^{\gamma}$-semistable sheaf.
Then
\begin{equation}\label{eq:BG}
\overline{\Delta}(v_H^\gamma(E)) \geq0.
\end{equation}
\end{Thm}

\begin{Rem}\label{rmk:PositiveChar}
The theory works similarly in finite characteristic. Inequalities similar to \eqref{eq:BG} proved by 
Langer in~\cite{Langer:positive} are  not sufficient to construct stability conditions.
Instead, by \cite[Theorem 1.3]{Koseki:BG}, there exists a constant $C_{X, H} \in \R_{\geq 0}$ such that inequality \eqref{eq:BG} holds if we add the term $C_{X,H} \cdot H^2/2$ to $\gamma$. This is sufficient for the construction of tilt-stability below.
\end{Rem}

\subsubsection*{Tilt-stability (Step 2)} Constructing tilt-stability from slope-stability needs two operations, see also fig.~\ref{fig:stepstotiltstability}. First, let $\beta \in \R$. We rotate the central charge $Z$ by setting $Z_{\beta}:=v_0+i(v_1-\beta\, v_0)$, and modify the abelian category accordingly to obtain $\Coh^\gamma_{H,\beta}(X)$ as follows:
\begin{align}
\cT^\beta & := \left\{ E \in \Coh(X) \colon \mu^{\gamma, -}_H(E) > \beta \right\} \nonumber\\
\cF^\beta & := \left\{ E \in \Coh(X) \colon \mu^{\gamma, +}_H(E) \le \beta \right\} \label{eq:defCohbeta} \\
\Coh^\gamma_{H,\beta}(X) &:= \left\{E\in\Db(X)\colon
\cH^l(E)=0,\text{ for  }l\neq0,-1, \cH^{-1}(E)\in \cF^\beta, \cH^0(E) \in \cT^\beta\right\}, \nonumber
\end{align}
where we denoted by $\mu_H^{\gamma,\pm}$ the first and last slope of the HN filtration with respect to $\mu_H^{\gamma}$-stability.
By tilting theory~\cite{Happel-al:tilting}, $\Coh^\gamma_{H,\beta}(X)$, the extension-closure of $\cF^\beta[1]$ and $\cT^\beta$, is the heart of a bounded t-structure on $\Db(X)$; in particular, it is an abelian category.

The pair $(Z_\beta, \Coh^{\gamma}_{H, \beta}(X))$ admits HN filtrations. The difference to slope stability in $\Coh(X)$ is subtle: torsion sheaves supported in codimension $\ge 2$ are in both categories, considered to have slope $+\infty$, and thus now have bigger phase than objects in $\cF^\beta[1]$.

For the second operation, we deform $Z_\beta$ while preserving the category $\Coh^{\gamma}_{H, \beta}(X)$. We follow the presentation in~\cite{Li:Quintic3fold} and denote by $U$ the open subset of $\R^2$ given by
\begin{equation} \label{eq:defU}
U:=\left\{(\alpha,\beta)\in\R^2\colon \alpha>\frac{\beta^2}{2}\right\}.
\end{equation}
For $(\alpha, \beta) \in U$ we consider the slope $\nu^\gamma_{H, \alpha, \beta}$ induced on
$\Coh^\gamma_{H, \beta}(X)$ by the central charge
\begin{equation} \label{eq:deftiltZ}
Z_{\alpha,\beta}(v):= - (v_2-\alpha\,v_0) + i\,(v_1-\beta\,v_0).
\end{equation}
Objects semistable with respect to  $v^\gamma_{H, \alpha, \beta}$ are called \emph{tilt-semistable}.
HN filtrations exist and Theorem~\ref{thm:BG} applies equally to tilt-semistable objects, which ensures that tilt-stability has a wall-and-chamber structure analogous to Bridgeland stability as we deform $(\alpha, \beta) \in U$.

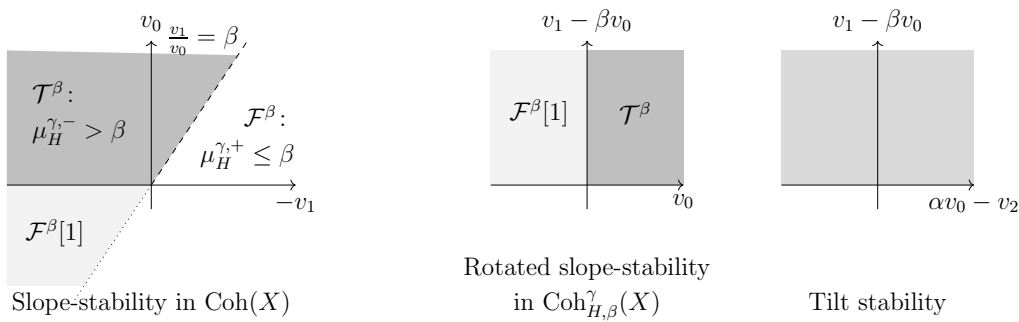
\begin{figure}
\begin{center}
\begin{minipage}{0.35\textwidth}
\scalebox{0.8}{
\begin{tikzpicture}[scale = 0.8]
\fill[gray!50] (-3,0)--(0,0)--(1.8,2.7)--(-3,2.8);
\fill[gray!10] (-3,0)--(0,0)--(-1.4,-2.1)--(-3, -2.1);
\node[align=left]  at (-1.5, 1.5) {$\cT^\beta\colon$\\ $\mu_H^{\gamma, -}> \beta$};
\node[align=right]  at (2, 1) {$\cF^\beta\colon$ \\ $\mu_H^{\gamma, +} \le \beta$};
\node[align=right]  at (-2, -1) {$\cF^\beta[1]$};
\draw[->] (-3,0)--(3,0) node[below]{$-v_1$};
\draw[->] (0,-0.5)--(0,3) node[above]{$v_0$\vphantom{$\beta$}};
\draw[dashed] (0,0)--(2, 3)  node[left]{$\frac{v_1}{v_0} = \beta$};
\draw[dotted] (0,0)--(-1.6, -2.4);
\node[align = center, anchor=south] at (0, -3) {Slope-stability in $\Coh(X)$\vphantom{$\Coh^\gamma_{H, \beta}(X)$}};
\end{tikzpicture}
}
\end{minipage}
\begin{minipage}{0.25\textwidth}
\scalebox{0.8}{
\begin{tikzpicture}[scale = 0.8]
\fill[gray!50] (0,2.8)--(0,0)--(2,0)--(2, 2.8);
\fill[gray!10] (0,2.8)--(0,0)--(-2,0)--(-2, 2.8);
\node  at (-1, 1.5) {$\cF^\beta[1]$};
\node  at (1, 1.5) {$\cT^\beta$};
\draw[->] (-2,0)--(2,0) node[below]{$v_0$};
\draw[->] (0,-0.5)--(0,3) node[above]{$v_1-\beta v_0$};
\node[align = center, anchor=south] at (0, -3) {Rotated slope-stability\\ in $\Coh^\gamma_{H, \beta}(X)$};
\end{tikzpicture}
}
\vfill
\end{minipage}
\begin{minipage}{0.25\textwidth}
\scalebox{0.8}{
\begin{tikzpicture}[scale = 0.8]
\fill[gray!30] (-2, 0)--(2, 0)--(2, 2.8)--(-2, 2.8);
\draw[->] (-2,0)--(2,0) node[below]{$\alpha v_0 - v_2$};
\draw[->] (0,-0.5)--(0,3) node[above]{$v_1-\beta v_0$};
\node[align = center, anchor = south] at (0, -3) {Tilt stability\vphantom{$\Coh^\gamma_{H, \beta}(X)$}};
\end{tikzpicture}
}
\end{minipage}
\end{center}

\caption{Rotating and deforming slope-stability to obtain tilt stability}
\label{fig:stepstotiltstability}
\end{figure}

\begin{Rem}\label{rmk:LePotier}
We can be more precise and sometimes obtain a larger set of tilt-stability conditions.
Following~\cite{FLZ:connected}, we define the \emph{Le Potier function} $\Phi_{X,H}^\gamma\colon\R \to \R$ as
\[
\Phi_{X,H}^\gamma(x):= \limsup_{\mu\to x} \left\{\frac{v_H^\gamma(E)_2}{v_H^\gamma(E)_0}\colon E\in\Coh(X) \text{ is }\mu_H^\gamma\text{-semistable with }\mu_H^\gamma(E)=\mu \right\}.
\]
It is upper semi-continuous, and by Theorem~\ref{thm:BG} it satisfies $\Phi_{X,H}^\gamma(x)\leq x^2/2$.
Tilt-stability is well-defined for all $(\alpha,\beta)\in\R^2$ such that $\alpha>\Phi_{X,H}^\gamma(\beta)$.
\end{Rem}

\begin{Ex}\label{ex:surface}
If $X$ is a surface, then $(Z_{\alpha,\beta},\Coh^\gamma_{H,\beta}(X))$ induces a Bridgeland stability condition on $\Db(X)$, as constructed in \cite{Aaron-Daniele}.
The support property  is given by the quadratic form~\eqref{eq:BG}.
If $X$ has finite Albanese morphism, then  all stability conditions on $\Db(X)$ with respect to $(v_H^\gamma,\Lambda^\gamma_H)$ are given by those constructed in Remark \ref{rmk:LePotier}, up to linear action~\cite{FLZ:connected}.
A stronger version of Theorem~\ref{thm:BG}, involving $\ch_1(E)^2$ rather than just $(H.\ch_1(E))^2$, gives a support  property with respect to the full numerical Grothendieck group of $\Db(X)$ and allows us to vary $H$; see~\cite[Theorem 3.5]{BMS:abelian3folds}.
In the case of a K3 surface $S$, the closure of the locus of such stability conditions and its translates under autoequivalences give the connected component $\Stab^\dagger(\Db(S))$ mentioned in Section~\ref{subsec:K3}.
\end{Ex}

\subsubsection*{The generalized Bogomolov--Gieseker inequality}
Let us assume $n\geq3$.
The main open question is to find an inequality involving $v_3$ for tilt-semistable objects, generalizing~\eqref{eq:BG}.
For $(\alpha,\beta)\in U$, let us define the quadratic form $\overline{Q}_{\alpha,\beta}$ on $\Lambda_{H,\leq3}^\gamma$ of signature $(2,2)$ by:
\begin{equation*}
\overline{Q}_{\alpha,\beta}(v) := \alpha\, \big( v_1^2 -2\, v_0\, v_2 \big) + \beta\, \big(3v_0\, v_3 -v_1\, v_2\big) + \big(2\, v_2^2 - 3\, v_1\, v_3 \big).
\end{equation*}

\begin{Def}\label{def:generalizedBG}
We say that $(X,H)$ satisfies the \emph{$\gamma$-generalized Bogomolov--Gieseker (BG) inequality} at $(\alpha,\beta)\in U$, if for all $E\in\Coh^\gamma_{H,\beta}(X)$ which are $\nu^\gamma_{H,\alpha,\beta}$-semistable, we have
\begin{equation}\label{eq:generalizedBG}
    \overline{Q}_{\alpha,\beta}(v_H^\gamma(E))\geq0.
\end{equation}
\end{Def}


\begin{Thm}[\cite{BMS:abelian3folds, PT15:bridgeland_moduli_properties}]
A polarized threefold satisfying the $\gamma$-generalized BG inequality at a point $(\alpha,\beta)\in U$ admits Bridgeland stability conditions. 
\end{Thm}

The construction of these stability conditions from tilt stability is completely analogous to the construction of tilt stability from slope-stability discussed above.
It was conjectured in~\cite[Conjecture~1.3.1]{BMT:3folds-BG} that all polarized threefolds satisfy the generalized BG inequality for all $(\alpha,\beta)\in U$ and $\gamma=e^{-B}$; this turned out to be too optimistic, see~\cite{Schmidt:counterexample}.
The following is a modification of the original conjecture, based on~\cite{Dulip:Fano,Li:Quintic3fold}:

\begin{Con}\label{conj:generalizedBG}
Let $(X,H)$ be a smooth complex projective polarized variety.
There exists a class $\gamma=\gamma_{X,H}$ as in~\eqref{eq:gamma} and an upper semi-continuous 
function $f=f_{X,H}^\gamma\colon\R\to\R$, such that $(X,H)$ satisfies the $\gamma$-generalized BG inequality for all $(\alpha,\beta)\in\R^2$ with $\alpha>f(\beta)$.
\end{Con}

Conjecture~\ref{conj:generalizedBG} has been established in a number of three-dimensional cases:
\begin{itemize}
    \item prime Fano threefolds~\cite{Macri:P3,Schmidt:Quadric,Li:FanoPic1}, with $\gamma=1$ and $f_{X,H}^\gamma(x)=x^2/2$;
    \item abelian threefolds~\cite{MacPiy:ab3folds,BMS:abelian3folds}, with $\gamma=e^{-B}$, for all $B\in\NS(X)_\R$, and $f_{X,H}^\gamma(x)=x^2/2$;
    \item the quintic threefold~\cite{Li:Quintic3fold}, with $\gamma=1$ and $f_{X,H}^\gamma(x)=x^2/2+(x-\lfloor x \rfloor)(\lfloor x \rfloor +1 - x)/2$;
    \item the complete intersection of quadratic
and quartic hypersurfaces in $\P^5$~\cite{ShengxuanLiu:CY}, with $\gamma=1$ and $f_{X,H}^\gamma(x)=x^2/2+(x-\lfloor x \rfloor)(\lfloor x \rfloor +1 - x)/2$;
    \item the blow-up of $\P^3$ at a point~\cite{Dulip:Fano}, with $H=-K_X/2$, $\gamma=(1,0,-\Gamma,0)$,
    \[
    \Gamma=\frac{1}{12}\left(c_2(X)-\frac{H\cdot c_2(X)}{H^3}H^2\right),
    \]
    and $f_{X,H}^\gamma(x)=x^2/2+(x-\lfloor x \rfloor +1)^2$; and
    \item threefolds with nef tangent bundle~\cite{BMSZ:Fano,koseki:2}, with $H$ any ample divisor, $\gamma=e^{-B}$, for all $B\in\NS(X)_\R$ (except in the case $X=\P(T_{\P^2})$, where $H=-K_X/2$ and $\gamma=1$), and $f_{X,H}^\gamma(x)=x^2/2$.
\end{itemize}

Similar versions have been proved for all Fano threefolds~\cite{BMSZ:Fano,Dulip:Fano} and for Calabi--Yau double and triple solids~\cite{koseki:CY}.
It is also known in some cases in arbitrary characteristic, e.g.~$\P^3$ (where $\gamma=1$, $f_{X,H}^\gamma(x)=x^2/2$).
There is no known counterexample to Conjecture~\ref{conj:generalizedBG} with $f_{X,H}^\gamma(x)=x^2/2$; a non-trivial choice of $\gamma$ is instead necessary: the blow-up of $\P^3$ at a point with the anti-canonical polarization does not satisfy the $\gamma$-generalized BG inequality if we take $\gamma=1$ and $f_{X,H}^\gamma(x)=x^2/2$~\cite{Schmidt:counterexample}.

\subsection{Tilt-stability methods}\label{subsec:LVLvsBN}

We  now describe three limit points of tilt-stability in the set $U$ defined in \eqref{eq:defU}, the \emph{small and large volume limit} and the \emph{Brill--Noether} point.
The latter two are  generalizations of limits discussed in Section~\ref{subsec:Curves} for Bridgeland stability conditions in the K3 surface case.

Throughout this section we fix a class  $v\in\Lambda_H^\gamma$, and assume $v_0 \neq 0$ for simplicity.

\subsubsection*{Walls}  We want to understand walls in $U$ for tilt-stability $(Z_{\alpha, \beta}, \Coh_{H,\beta}(X))$, defined by \eqref{eq:deftiltZ} and \eqref{eq:defCohbeta}, of objects of class $v$. We may assume $\bar\Delta(v) \geq 0$ and define
\[
p(v):=\left(\frac{v_2}{v_0},\frac{v_1}{v_0} \right) \in \R^2 \setminus U.
\]
Given a line $L$ containing $p(v)$, we define the \emph{potential wall} in $U$ for $v$ associated to $L$ by
\[
\cW_L:=\left\{(\alpha,\beta)\in U\cap L\right\}.
\]

The main property of walls are the following, see also Figure~\ref{fig:U}. Two objects of $\Coh^\gamma_{H, \beta}(X)$ with classes $v$ and $w$ have the same slope with respect to $\nu^\gamma_{\alpha, \beta, H}$ if and only if $(\alpha, \beta)$ lies on the line passing through $p(v)$ and $p(w)$.
If an object of class $v$ is tilt-(semi)stable at one point of 
$\cW_L$, then it is tilt-(semi)stable for all points in $\cW_L$. Moreover,
actual walls for $v$ are a locally finite set of potential walls, and 
tilt-stability 
is unchanged except when crossing one of these walls. 

The following elementary observation is often useful for induction arguments:

\begin{Rem}\label{rmk:DeltaDecreases}
Given a wall $\cW$ for $v$ in $U$,  let $w_1,\dots,w_m$ be the classes of the Jordan--H\"older factors of a  tilt-semistable object with class $v$ at a point of $\cW$. Then the version of Theorem~\ref{thm:BG} for tilt-stable objects  implies that $\overline{\Delta}(w_l)\leq\overline{\Delta}(v)$, for all $l=1,\dots,m$, with equality if and only if all $(w_l)_{\leq 2}$ and $v_{\leq 2}$ are proportional and $\overline{\Delta}(w_l)=\overline{\Delta}(v)=0$.
In particular, the structure sheaf $\cO_X$ or its shift $\cO_X[1]$ is tilt-stable everywhere in $U$.
\end{Rem}

\begin{Ex}
The equation $\overline{Q}_{\alpha,\beta}(v)=0$ defines a line containing $p(v)$. If it intersects $U$, we call the associated potential wall  the \emph{BG wall}, which  gives a bound on  walls for $v$.
\end{Ex}

\subsubsection*{The small volume limit point}
Assume that $\overline{\Delta}(v)>0$.
We define the two \emph{small volume limit points} $\overline{\beta}(v)$ as the points $(\beta^2/2,\beta)\in\R^2$ where the tangent to the parabola contains $p(v)$.

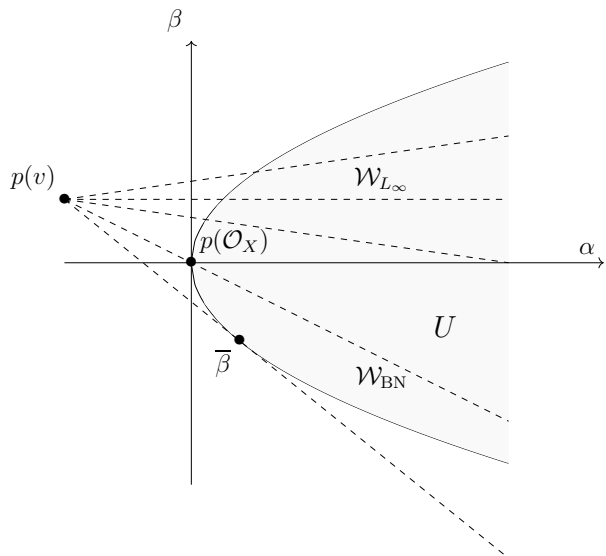
\begin{figure}
\begin{center}
\scalebox{0.8}{
\begin{tikzpicture}[x=30pt,y=30pt]

\draw[black] plot[smooth,samples=100,domain=0:5] (\x, {sqrt(2*\x)});
\draw[black] plot[smooth,samples=100,domain=0:5] (\x, -{sqrt(2*\x)});

\fill[gray!5!white, domain=0:5, variable=\x]
      (0, 0)
      -- plot ({\x}, {sqrt(2*\x)})
      -- (5, 0)
      -- cycle;

\fill[gray!5!white, domain=0:5, variable=\x]
      (0, 0)
      -- plot ({\x}, -{sqrt(2*\x)})
      -- (5, 0)
      -- cycle;
      
\draw[->] (-2,0) -- (6.5,0)node[above left] {$\alpha$};
\draw[->] (0,-3.5)-- (0,3.5) node[above left] {$\beta$};

\node at (0,0) {$\bullet$};
\node[above right] at (0, 0) {$p(\cO_X)$};
\node at (-2,1) {$\bullet$};
\node[above left] at (-2,1) {$p(v)$};
\node at (4,-1) {\large $U$};

\draw[dashed] (-2,1) -- (3,1)node[above] {$\cW_{L_\infty}$} -- (5,1);
\draw[dashed] (-2,1) -- (5,2);
\draw[dashed] (-2,1) -- (5,0);
\draw[dashed] (-2,1) -- (3,-1.5)node[below] {$\cW_{\mathrm{BN}}$} -- (5,-2.5);
\draw[dashed] (-2,1) -- (0.76,-1.23)node {$\bullet$} node[below left] {$\overline{\beta}$} -- (5,-4.66);

\end{tikzpicture}
}
\end{center}
\caption{Walls for $v$} \label{fig:U}
\end{figure}

\noindent
By the local finiteness of walls, they can accumulate only at the small volume limit points, and only finitely many lie outside a neighborhood of them.
Objects which are tilt-stable in a neighborhood of a small volume limit point are called \emph{$\overline{\beta}$-stable} and have strong vanishing properties, which are useful both in proving cases of Conjecture~\ref{conj:generalizedBG} or in applications of it.
For instance, if Conjecture~\ref{conj:generalizedBG} is true and the BG wall exists, there are no $\overline{\beta}$-stable objects for the corresponding small limit point. 
In the approach mentioned in the Introduction, the small limit point is a promising choice for a point of interest.

\subsubsection*{The large volume limit point}
As in Figure~\ref{fig:U}, consider the horizontal line
\[
L_{\infty}:=\left\{ \beta = v_1/v_0 \right\}.
\]
It is called the \emph{large volume limit wall} for $v$ and a point in it is called a \emph{large volume limit point}.
Object which are tilt-stable on a point in $U$ nearby the large volume limit wall---equivalently, for $\alpha \leadsto +\infty$---essentially correspond to Gieseker-stable sheaves or derived duals of them, according to which side of the wall we are. 

\subsubsection*{The Brill--Noether point}
The last point we are interested in helps studying global sections of objects.
In the K3 surface case, as discussed in Section~\ref{subsec:Curves}, this is where applications to Brill--Noether theory come from.
In tilt-stability, it is given by the potential wall $\cW_{\mathrm{BN}}$ associated to the line passing through $p(v)$ and $(0,0)$, see again Figure~\ref{fig:U}.
We call the point $(0,0)$ the \emph{Brill--Noether (BN) point}.

\subsection{Tilt-stability applications}\label{subsec:applications}

In this section, we give an informal exposition of three applications of tilt-stability and Conjecture~\ref{conj:generalizedBG}.

\subsubsection*{Curves on threefolds}\label{subsec:genus}

Let $X=\P^3$.
We want to study the following question, called the \emph{Halphen problem}: what is the maximal genus $g$ of an integral curve in $\P^3$ of degree $d$ which is not contained in a surface of degree $< k$? While the question is open for smaller $d$, a celebrated theorem of Gruson--Peskine and Harris~\cite{GP:genus,Harris:space-curves} gives such maximal genus $G(d,k)$, when $d>k(k-1)$
We give an idea how to reprove this theorem by using tilt-stability.
The approach works for any threefold which satisfies Conjecture~\ref{conj:generalizedBG} and a few extra assumptions, see~\cite[Theorem 1.2]{MS:genus} for the precise statement. For instance, it yields new results for principally polarized abelian threefolds of Picard rank~1.

Let $C$ be a curve as above whose genus is larger than the expected bound $G(d,k)$.
We look at the twisted ideal sheaf $\mathscr{I}_C(k-1)$ and we let $v=v(\mathscr{I}_C(k-1))$; here $\gamma=1$ and the generalized BG inequality holds, for all $(\alpha,\beta)\in U$.
The first step is a straightforward application of Conjecture~\ref{conj:generalizedBG}, which shows that in a neighborhood of the small volume limit point, there are no tilt-semistable objects with class $v$.
The second step is to use the information that $C$ is not contained in a surface of degree $k-1$, which says that the BN wall does not give a wall for $\mathscr{I}_C(k-1)$.
To summarize: $\mathscr{I}_C(k-1)$ is tilt-stable at the large volume limit, it must be destabilized at a certain wall, which cannot be on the BN wall.

To get a contradiction we need  to analyze the finitely many walls for $\mathscr{I}_C(k-1)$.
This is where the assumption $d>k(k-1)$ comes in.
In fact, possible destabilizing subobjects for $\mathscr{I}_C(k-1)$ are always reflexive sheaves: in the range $d>k(k-1)$ they have rank either~1 or~2.
The rank~1 case, namely invertible sheaves, can be dealt with by the fact that the BN wall is not an actual wall.
For the rank~2 case, we use once more a similar strategy to get bounds on the third Chern character of such rank~2 sheaves and thus a contradiction, since we have control on the discriminant, by Remark~\ref{rmk:DeltaDecreases}.

\subsubsection*{Higher rank DT invariants on CY3s}\label{subsec:DT}

Let $(X,H)$ be a complex polarized Calabi--Yau threefold.
In a recent sequence of papers~\cite{FT1,FT2}, Feyzbakhsh and Thomas proved the following theorem: if $(X,H)$ satisfies the generalized BG inequality on $U$, then the higher rank Donaldson--Thomas (DT) theory is completely governed by the rank~1 theory, i.e., Hilbert schemes of curves.
There is some flexibility on the assumption on the generalized BG inequality; in particular, their theorem holds for the examples of Calabi--Yau threefolds where Conjecture~\ref{conj:generalizedBG} has been proved, e.g.~for the quintic threefold.

\def\DT{\mathrm{DT}}
The fundamental idea is as follows. Fix a class $v$ of rank $r$. The  DT invariant $\DT_{\text{large-v}}(v)$ of tilt-semistable objects of class $v$ near the large-volume limit is essentially the classical DT invariant. Now let $n \gg 0$. The role of the point of interest is first played by a variant of the Brill--Noether wall: the Joyce--Song wall $\cW_\mathrm{JS}$ where $v$ and $\cO_X(-n)[1]$ have the same slope. This is a wall not for $v$ (where it is contained in the large-volume chamber), but for the class $v^{(n)} = v - v(\cO_X(-n))$ of rank $r-1$: there are objects $E$ of class $v^{(n)}$ destabilised by a short exact sequence $F \to E \to \cO_X(-n)[1]$ with $F \in M_{\alpha, \beta}(v)$. They are generically parameterised by a projective bundle over the DT moduli space for $v$.

Now consider the DT invariant $\DT_{(\alpha, \beta)}(v^{(n)})$ as $(\alpha, \beta)$ moves on a path from the small-volume limit---our point of interest---for $v^{(n)}$ to its large-volume limit. Conjecture~\ref{conj:generalizedBG} shows $\DT_{\text{small-v}}(v^{(n)}) = 0$. Applying the generalised BG inequality again, the authors show that \emph{except for the Joyce--Song wall}, all other walls are defined by sheaves of rank~$\leq r-1$. Thus, the Joyce--Song wall-crossing formula in \cite{Joyce-Song} gives a relation of the form
\begin{align*}
\DT_{\textrm{large-v}}(v^{(n)}) & =
\DT_{\textrm{small-v}}(v^{(n)}) + \text{Wall-crossing contributions}
\\ &= 0 + \mathrm{WallCr}(\text{lower rank DT invariants})
+ \chi(\cO_X(-n), v) \cdot \DT_{\text{large-v}}(v).
\end{align*}
This shows that $ \DT_{\text{large-v}}(v)$ is determined by lower rank DT invariants.

\subsubsection*{The quintic threefold and Clifford-type bounds}\label{subsec:quintic}

In the proof of the generalized BG inequality for the quintic threefold in~\cite{Li:Quintic3fold}, the first idea is the following: if we know the generalized BG inequality at the BN point, we know it everywhere (for an appropriately chosen function $f$ in Conjecture~\ref{conj:generalizedBG}).
Here, Li uses the same idea as in Section~\ref{subsec:Curves}: a stronger \emph{classical} Bogomolov--Gieseker type inequality for the quotient $Q$ in \eqref{eq:BNses} implies a bound for $h^0(E)$, and consequently for $\chi(E)$ and thus $\ch_3(E)$. The stronger bound for $Q$ is deduced by a restriction theorem from stronger bounds on $(2, 5)$-complete intersection surfaces. Using the logic of Section~\ref{subsec:Curves} in reverse, this bound is reduced to Clifford-type bounds for stable vector bundles on $(2, 2,5)$-complete intersection curves $C$. Now we consider the embedding $C \subset S_{2, 2}$ of $C$ into the $(2,2)$-complete intersection del Pezzo surface $S_{2, 2}$. Riemann-Roch directly implies a stronger Bogomolov--Gieseker inequality on
$S_{2, 2}$, and shifting the logic of Section~\ref{subsec:Curves} back to forward gear implies the desired Clifford bounds.

These Clifford bound arguments yield new results even for planar curves \cite{FL:Clifford}.


\section{Further research directions}\label{sec:Further}

A proof of Conjecture~\ref{conj:generalizedBG},  and thus the existence of Bridgeland stability conditions on threefolds, would evidently be tremendous progress. We present here some more specific open questions related to the topic of our survey.

\subsubsection*{The quintic threefold and Toda's conjecture}
The picture for Bridgeland stability conditions is not yet complete, even for quintic threefolds.
The expectation in~\cite{Aspinwall:Dbranes-CY} (see also \cite[Remark 3.9]{Bri:ICM} for more details) is the following: there exists a closed embedding
\[
I\colon \mathfrak{M}_K \to \left[ \mathrm{Aut}(\Db(X))\backslash \Stab(\Db(X))/\C \right],
\quad \text{where} \quad \mathfrak{M}_K:=\left[ \{\psi\in\C\,:\,\psi^5\neq1\} / \mu_5\right]
\]
and $I(\psi)=(Z_\psi,\cP_\psi)$ where $Z_\psi$ is a solution of the associated Picard--Fuchs equation (see~\cite[Section~3.2]{Tod:Gepner}).
Li's Theorem in~\cite{Li:Quintic3fold} describes only a neighborhood of the large volume limit. The global picture would follow from an appropriate answer to the following question.
\begin{Ques} Is there a better bound for the
Le Potier function for quintic threefolds (extending \cite[Conjecture~1.2]{Tod:Gepner} for slope $\frac 12$ to arbitrary slope), and a version of Conjecture \ref{conj:generalizedBG} with $f(x)$ approximating this Le Potier function?
\end{Ques}

\subsubsection*{Higher dimension}
The first step towards stability conditions in higher dimension would be the following:
\begin{Ques}
Can we prove  Conjecture~\ref{conj:generalizedBG} by induction on the dimension of $X$ once it is known for (a suitable class of) threefolds?
\end{Ques}

\subsubsection*{Moduli spaces of polarized non-commutative varieties}

A natural extension of the results in Section~\ref{subsec:K3} would be to answer the following:

\begin{Ques}\label{ques:HK}
Is any smooth projective polarized irreducible holomorphic symplectic variety $(M,H)$ of $\mathrm{K3}^{[n]}$-type isomorphic to the moduli space $(M_\sigma(v),\ell_\sigma)$ of $\sigma$-stable objects for a stability condition $\sigma$ on a K3 category $\cD$?
\end{Ques}

A candidate for $\cD$ has been constructed in~\cite{MM:K3category}, over an open subset of the moduli space. It is not known if it can be realized as an admissible subcategory. 
The question is closely related to the following, completing the analogy between stability conditions and polarizations:

\begin{Ques}
Is there an algebraic moduli space of polarised non-commutative K3 surfaces, parameterising pairs $(\cD, \sigma)$ where $\cD$ is a K3 category deformation--equivalent to 
$\Db(S)$ for a projective K3 surface $S$, and $\sigma$ is a stability condition on $\cD$?
\end{Ques}

The theory of stability conditions in families developed in~\cite{families} provides a definition of morphisms to this moduli space.

\subsection*{Acknowledgements}
The work reviewed here benefited immensely from the input of our mentors, collaborators and friends.  Aaron Bertram envisioned many of the topics surveyed here and persistently encouraged us to explore them. Many of our (current or former) mentees carried these topics much further than we imagined.
We want to thank in particular Enrico Arbarello, Claudio Bartocci, Marcello Bernardara, Tom Bridgeland, Ugo Bruzzo, Fran\c{c}ois Charles, \.{I}zzet Co\c{s}kun, Alastair Craw, Olivier Debarre, Gavril Farkas, Tommaso de Fernex, Soheyla Feyzbakhsh, Laure Flapan, Brendan Hassett, Jack Huizenga, Daniel Huybrechts, Qingyuang Jiang, Naoki Koseki, Alexander Kuznetsov, Mart\'i Lahoz, Chunyi Li, Antony Maciocia, Yuri I.~Manin, Eyal Markman, Sukhendu Mehrotra, Howard Nuer, Kieran O'Grady, Tony Pantev, Alexander Perry, Laura Pertusi, Giulia Sacc\`a, Benjamin Schmidt, Paolo Stellari, Richard Thomas, Yukinobu Toda, Claire Voisin, Michael Wemyss, K\=ota Yoshioka, Ziyu Zhang and Xiaolei Zhao for their invaluable constant support.


\end{document}